\documentclass[12pt]{article}
\usepackage[a4paper,margin=3cm]{geometry}
\usepackage{amsmath}
\usepackage{amssymb}

\newtheorem{theorem}{Theorem}[section]
\newtheorem{lemma}[theorem]{Lemma}

\newtheorem{proposition}[theorem]{Proposition}

\newenvironment{proof}{{\bf Proof.}}{\hfill$\Box$\\}
\newenvironment{definition}{{\vskip 3ex\bf Definition.} }{\\}
\newenvironment{remark}{{\vskip 3ex\bf Remark.} }{\\}

\newcommand{\R}{\mathbb{R}}

\newcommand{\ad}{\mathrm{ad}}
\newcommand{\rk}{\mathrm{rk}}

\newcommand{\spn}{\mathrm{span}}

\newcommand{\XX}{\mathcal{X}}

\newcommand{\HH}{\mathcal{H}}
\newcommand{\CC}{\mathcal{C}}
\newcommand{\DD}{\mathcal{D}}
\newcommand{\EE}{\mathcal{E}}
\newcommand{\BG}{{\bf G}}
\newcommand{\BF}{{\bf F}}
\newcommand{\BV}{{\bf V}}
\newcommand{\BW}{{\bf W}}
\newcommand{\BX}{{\bf X}}
\newcommand{\BY}{{\bf Y}}
\newcommand{\BH}{{\bf H}}
\newcommand{\vp}{\varphi}

\title{\bf Canonical frames for $Gl(2)$-structures}
\author{Wojciech Kry\'nski\thanks{
Institute of Mathematics, Polish Academy of Sciences, ul.~\'Sniadeckich 8, 00-956 Warszawa, Poland, E-mail: krynski@impan.gov.pl.\newline
Supported by Junior Research Fellowship from Erwin Schr\"odinger Institute, Vienna.
}}

\begin{document}
\maketitle

\begin{abstract}
We construct a canonical frame for an arbitrary $Gl(2)$-structure thus solving the equivalence problem for $Gl(2)$-structures. Our treatment includes also a problem of contact equivalence of ordinary differential equations and applies to certain classes of vector distributions. Additionally we characterise $Gl(2)$-structures which are defined by ODEs.
\end{abstract}

{\bf Keywords:} $Gl(2)$-structure, frame bundle, problem of equivalence, ordinary differential equations, distribution.

{\bf MSC:} 53A55, 34A26

\section{Introduction}\label{sec1}
In the present paper we solve the equivalence problem for $Gl(2)$-structures. A $Gl(2)$-structure on a manifold $M$ is a frame bundle on $M$ with structure group $Gl(2)$ which acts irreducibly on fibres of the bundle. It follows that a splitting
$$
TM=\underbrace{E\odot\cdots\odot E}_{\dim M-1}
$$
is given for a certain rank two bundle $E$ over $M$, where $\odot$ is a symmetric tensor product. Equivalently, for any $x\in M$ we can identify the tangent space $T_xM$ and the space of homogeneous polynomials of order $\dim M-1$ in two variables with natural $Gl(2)$-action.

In the recent years, $Gl(2)$-structures (known also under the name \emph{paraconformal structures}) attract much attention due to their connections to ODEs. The first result in this direction goes back to S-S.~Chern \cite{Ch} who considered equations of order 3 and proved that if an equation satisfies the so-called W\"unschmann condition then it defines a canonical $Gl(2)$-structure on its solution space (in this case it is just a conformal Lorentz metric). The results in higher dimensions were obtained by R.~Bryant \cite{B}, M.~Dunajski, P.~Tod \cite{DT} and others \cite{DG,GN,FKN,K1}. Besides, $Gl(2)$-structures appear also in the theory of distributions. For instance, in the paper \cite{K2} a link was found between $Gl(2)$-structures and distributions of type $(3,5,6)$ with parabolic symbol, whereas in the paper \cite{KZ} a link was found between $Gl(2)$-structures and $(2k+1,2k+3)$-distributions with the so-called maximal first Kronecker index and integrable square root. 

In the case of dimension 3 all  germs of conformal Lorentz metrics on three dimensional manifolds can be obtained from ODEs \cite{FKN,K1}. Therefore the problem of equivalence of $Gl(2)$-structures on manifolds of dimension 3 is reduced to the problem of equivalence of ODEs of order 3. This problem was solved by Chern \cite{Ch} who has constructed a Cartan connection taking values in $\mathfrak{sp}(4,\R)$ (see also \cite{DKM}). On the contrary, if the dimension of an underlying manifold is greater than 3 then a generic $Gl(2)$-structure cannot be realised as an equation \cite{K1}. Therefore the problem of equivalence of $Gl(2)$-structures is more general in this case. We solve it in the present paper by constructing a canonical frame on a certain principal bundle and proving:
\vskip 1ex
\noindent\emph{Two $Gl(2)$-structures are equivalent if and only if the corresponding frames are equivalent.}
\vskip 1ex
This gives the complete answer to the problem of equivalence of $Gl(2)$-strucutres since the problem of equivalence of frames has a well known solution due to E.~Cartan (see \cite{O}).

It is worth to mention here that in fact we solve an equivalence problem for more general objects than $Gl(2)$-structures (see Section \ref{sec2}). As a by-product, we provide a new solution to the problem of contact equivalence of ODEs (not only for those ODEs which define $Gl(2)$-structures). In this case the coframe dual to the canonical frame defines a Cartan connection. Our results also apply to above-mentioned distributions of type $(3,5,6)$ or $(2k+1,2k+3)$ (we refer to \cite{K2} and \cite{KZ} for details).

\section{Preliminaries}\label{sec2}
A distribution $\DD$ on a manifold $M$ is a smooth subbundle of the tangent bundle. If $\DD$ has rank $m$ then in a neighbourhood of any $x\in M$ there exist vector fields $X_1,\ldots,X_m$ such that
$$
\DD=\spn\{X_1,\ldots,X_m\}
$$
and $(X_1,\ldots,X_m)$ is called \emph{local frame} of $\DD$. The set of smooth sections of $\DD$ is denoted $\Gamma(\DD)$. If $X\in \Gamma(\DD)$ then $X$ is a vector field tangent to $\DD$. If $\DD$ and $\EE$ are two distributions then we define their Lie bracket at a point $x\in M$ by the formula
$$
[\DD,\EE](x)=\spn\{[X,Y](x)\ |\ X\in\Gamma(\DD),\ Y\in\Gamma(\EE)\},
$$
where, on the right hand side $[X,Y]$ is the Lie bracket of vector fields. The definition implies that $\DD(x),\EE(x)\subset [\DD,\EE](x)$. Further, we  set $\ad_\EE\DD=[\EE,\DD]$ and define by induction
$$
\ad^{i+1}_\EE\DD=[\EE,\ad^i_\EE\DD].
$$
If $\dim\ad^i_\EE\DD(x)$ is a constant function of $x$ then $\ad^i_\EE\DD$ is a new distribution on $M$. In \cite{K1} we have introduced the following notion
\begin{definition}
A pair $(\XX,\DD)$ of two distributions on a manifold $M$ is \emph{regular} if $\rk\,\XX=1$, $\rk\,\DD=2$, $\XX\subset\DD$ and the following two conditions
\begin{enumerate}
\item[(G1)] $\rk\,\ad_\XX^i\DD=i+2$ for $i=1,\ldots,k$.
\item[(G2)] $\ad_\XX^k\DD=TM$.
\end{enumerate}
are satisfied.
\end{definition}

We will say that two regular pairs are \emph{equivalent} if there exists a diffeomorphism which transforms one pair onto the other. As we will see below, on the one hand regular pairs generalise the notion of ODEs and on the other hand they generalise the notion of $Gl(2)$-structures.
\vskip 1ex
{\bf Ordinary differential equations.}
If $(F)$ is an equation of order $k+1$ given in the form
$$
x^{(k+1)}=F(t,x,x',\ldots,x^{(k)})
$$
then we can consider $M=J^k(\R,\R)$, the space of $k$-jets, with the canonical rank two Cartan distribution $\CC$ and the rank one distribution $\XX_F$ spanned by the total derivative
$$
X_F=\partial_t+x_1\partial_{x_0}+\cdots+x_k\partial_{x_{k-1}} +F(t,x_0,\ldots,x_k)\partial_{x_k},
$$
where $(t,x_0,\ldots,x_k)$ are natural coordinates on $J^k(\R,\R)$. It is an easy task to check that the pair $(\XX_F,\CC)$ is regular. Moreover two equations are contact equivalent if and only if the corresponding regular pairs are equivalent. Regular pairs equivalent (locally) to $(\XX_F,\CC)$ will be called \emph{of equation type}.
\vskip 1ex
{\bf W\"unschman condition and $Gl(2)$-structures.}
The W\"unschman condition for a general ODE was introduced in \cite{DT}. It is equivalent to the vanishing of the Wilczynski-Doubrov invariants \cite{D}. In \cite{K1} we have shown that both Wilczynski invariants and W\"unschmann condition can be defined also for regular pairs. Namely, a regular pair $(\XX,\DD)$ satisfies the W\"unschmann condition if there exists a local frame $(X,V)$ of $\DD$ such that $\XX=\spn\{X\}$ and $\ad^{k+1}_XV=0\mod \XX$. We have proved the following
\begin{theorem}\label{t1}
There is one-to-one correspondence between germs of $Gl(2)$-structures and germs of regular pairs satisfying W\"unschmann condition.
\end{theorem}
A direct, geometric, construction relating $Gl(2)$-structures and regular pairs is given in \cite{K1} and we refer there for the proof of Theorem \ref{t1} above. Let us note here that if a $Gl(2)$-structure is defined on a manifold of dimension $k+1$ then the corresponding regular pair is defined on a manifold of dimension $k+2$ (where $k$ is a parameter from the definition of regular pairs).
\vskip 1ex
In the next sections we will consider the problem of equivalence of regular pairs. It follows from above that this problem contains both: the problem of equivalence of $Gl(2)$-structures and the problem of contact equivalence of ODEs. But, it is more general than both of them, provided that $k>2$. On the other hand, as we mentioned in the Introduction, the problem of equivalence of $Gl(2)$-structures on three dimensional manifolds has known solution. Therefore we will assume that $k>2$.

\section{Canonical bundle}\label{sec3}
In \cite{K1} we have proved the following
\begin{proposition}\label{p1}
For a given regular pair $(\XX,\DD)$ there exists a local frame $(X,V)$ of $\DD$ such that $\XX=\spn\{X\}$, $\DD=\spn\{X,V\}$ and
\begin{equation}\label{eq1}
\ad_X^{k+1}V=0\mod X,V,\ad_XV,\ldots,\ad_X^{k-2}V.
\end{equation}
If $(X',V')$ is a different frame of $\DD$ satisfying \eqref{eq1} then there exist functions $f,g$ on $M$ such that $X'=fX$, $V'=gV\mod X$ and
\begin{eqnarray}
&&fX(g)=-\frac{k}{2}X(f)g,\label{eq2}\\
&&2fX^2(f)-X(f)^2=0.\label{eq3}
\end{eqnarray}
\end{proposition}
\begin{proof}
See \cite[Proposition 4.1]{K1}
\end{proof}
\begin{remark}
Let $\R\ni t\mapsto\gamma(t)\in M$ be a curve such that $X=\gamma_*\left(\frac{\partial}{\partial t}\right)$ and let $\R\ni t'\mapsto \vp(t')\in\R$ be such that $X'=\gamma_*\circ\vp_*(\frac{\partial}{\partial t'})$. Then \eqref{eq3} takes the form
$$
2\frac{\dddot\vp}{\dot\vp}-3\left(\frac{\ddot\vp}{\dot\vp}\right)^2=0,
$$
i.e. $\mathbb{S}(\vp)=0$ where $\mathbb{S}$ stands for Schwartz derivative. It follows that $t$ and $t'$ are related by M\"obius transformation
$$
t'=\frac{at+c}{bt+d}
$$
and thus we get $Gl(2)$-action on the set of solutions of \eqref{eq3}.
\end{remark}
\begin{definition}
Let $(\XX,\DD)$ be a regular pair and assume that $X\in\Gamma(\XX)$ and $V\in\Gamma(\DD)$ satisfy conditions of Proposition \ref{p1}. Then $X$ is called \emph{projective} vector field and $V$ is called a \emph{normal} vector field corresponding to $X$.
\end{definition}

Note that it follows from \eqref{eq3} that the set of projective vector fields is 2-parameter family (as a solution to ODE of order 2). Moreover, if we denote by $j^1\XX(x)$ the set of 1-jets of sections of $\XX$ at $x\in M$ then there is a natural bijection between $j^1\XX(x)$ and the set of projective vector fields restricted to a segment of the integral line of $\XX$ passing through $x$. We immediately get from the Remark above that the affine group $Aff(1)$ acts regularly on $j^1\XX(x)$, where
$$
Aff(1)=\left\{\ \left(\begin{array}{cc} a & b \\ 0 & 1\end{array}\right)\in Gl(2)\ |\ a,b\in\R,\ a\neq 0 \right\}.
$$

We are now in the position to construct a canonical principal bundle for a regular pair $(\XX,\DD)$ on a manifold $M$. Let $x\in M$ and
$$
E(\XX,\DD)(x)=\left(\DD(x)/\XX(x)\right)\times j^1\XX(x)
$$
In above $\DD(x)/\XX(x)$ is a quotient linear space of dimension 1 and $Gl(1)$ acts on $\DD(x)/\XX(x)$. We define
$$
E(\XX,\DD)=\bigcup_{x\in M}E(\XX,\DD)(x),
$$
and it follows that $E(\XX,\DD)$ is a principal $Gl(1)\oplus Aff(1)$-bundle over $M$. We call it \emph{the canonical bundle} of $(\XX,\DD)$. Note that the group $Gl(1)\oplus Aff(1)$ is isomorphic to the group $T(2)\subset Gl(2)$ of upper triangular matrices. The isomorphism is given by the formula
$$
Gl(1)\oplus Aff(1)\ni\left((c),\left(\begin{array}{cc}a & b \\ 0 & 1\end{array}\right)\right) \mapsto \left(\begin{array}{cc}ca & cb \\ 0 & c\end{array}\right)\in T(2).
$$
The canonical bundle $E(\XX,\DD)$ carries additional structures.
\vskip 1ex
{\bf Projections.}
First of all we have the projection to the base manifold
$$
\pi\colon E(\XX,\DD)\to M.
$$
Besides, we have the following natural projections
$$
\pi_D\colon E(\XX,\DD)\to \DD/\XX,\qquad \pi_X\colon E(\XX,\DD)\to\XX,\qquad\pi_X^1\colon E(\XX,\DD)\to j^1\XX.
$$
\vskip 1ex
{\bf Fundamental vector fields.}
Since $Gl(1)\oplus Aff(1)$ acts on fibres of $E(\XX,\DD)$, any vector $\vec a$ in Lie algebra $\mathfrak{gl}(1)\oplus \mathfrak{aff}(1)$ defines a fundamental vector field, denoted $\mathbf{A}$, on $E(\XX,\DD)$. We will distinguish the following fundamental vector fields
\begin{enumerate}
\item $\BG$ - fundamental vector field corresponding to $(1)\in\mathfrak{gl}(1)$,
\item $\BF^0$ - fundamental vector field corresponding to $\left(\begin{array}{cc}1 & 0 \\ 0 & 0\end{array}\right)\in\mathfrak{aff}(1)$,
\item $\BF^1$ - fundamental vector field corresponding to $\left(\begin{array}{cc}0 & 1 \\ 0 & 0\end{array}\right)\in\mathfrak{aff}(1)$.
\end{enumerate}
We will abbreviate $\BF=(\BF^0,\BF^1)$. Note that $\BF^1\in\ker{\pi_D}_*\cap\ker{\pi_X}_*$.
\vskip 1ex
{\bf Partial Ehresmann connection.}
There is a natural rank one distribution $\widetilde{\XX}$ on $E(\XX,\DD)$ such that $\pi_*(\widetilde{\XX})=\XX$. Namely, it follows from Proposition \ref{p1} (equations \eqref{eq2}-\eqref{eq3}) that any point $p\in E(\XX,\DD)$ can be uniquely extended to a curve $t\mapsto p(t)$ such that $p(0)=p$, $t\mapsto X(t)=\pi_X(p(t))$ is a projective vector field along the integral line of $\XX$ passing through $\pi(p(0))$ and $t\mapsto V(t)=\pi_D(p(t))\mod\XX$ is a corresponding normal vector field along the integral line of $\XX$. We define $\widetilde{\XX}(p)=\spn\{\dot{p}(0)\}$.

\section{Canonical frame}\label{sec4}
Our aim is to choose additional vector fields: $\BX$ and $\BV^i$, $i=0,\ldots,k$ such that the tuple 
$$
(\BG,\BF^0,\BF^1,\BX,\BV^0,\ldots,\BV^k)
$$
constitutes a frame on $E(\XX,\DD)$.
\begin{definition}
A frame $(\BG,\BF^0,\BF^1,\BX,\BV^0,\ldots,\BV^k)$ on $E(\XX,\DD)$ is \emph{adapted} if the following conditions
\begin{enumerate}
\item $\pi_*(\BV^0(p))=\pi_D(p)$ for any $p\in E(\XX,\DD)$,
\item $\pi_*(\BX(p))=\pi_X(p)$ and $\BX(p)\in \widetilde{\XX}$ for any $p\in E(\XX,\DD)$,
\item $\BV^i=\ad^i_\BX\BV^0$ for any $i=1,\ldots,k$,
\end{enumerate}
are satisfied.
\end{definition}

Note that the vector field $\BX$ is uniquely defined by condition (2) above, since $\widetilde{\XX}$ is a rank one distribution and $\pi_*(\BX(p))=\pi_X(p)$ chooses a unique vector in it. In order to define $\BV^i$ uniquely we will make a normalisation. For this we introduce coefficients $T^{pq}_r$ in the following way
$$
[\BV^p,\BV^q]=\sum_{r=0}^kT^{pq}_r\BV^r\mod \BX,\BG,\BF.
$$
\begin{remark}
If one defines $\widetilde{\HH}=\spn\{\BX,\BV^0,\ldots,\BV^k\}$ then $\widetilde{\HH}$ is a distribution transversal to fibres of $\pi\colon E(\XX,\DD)\to M$, i.e.\ an Ehresmann connection. Then $T^{pq}_r$ can be called torsion coefficients.
\end{remark}

Now we are in position to prove our main result.
\begin{theorem}\label{t2}
Let $(\XX,\DD)$ be a regular pair on a manifold $M$ of dimension $n=k+2$, where $k\geq 3$. There exists the unique adapted frame on $E(\XX,\DD)$ satisfying conditions
\begin{equation}\label{eq4}
T^{01}_0=0,\qquad T^{01}_1=0,\qquad T^{01}_2=0,\qquad T^{03}_3=0
\end{equation}
Two pairs $(\XX,\DD)$ and $(\XX',\DD')$ are equivalent if and only if the corresponding frames on $E(\XX,\DD)$ and $E(\XX',\DD')$ are diffeomorphic.
The group of symmetries of $(\XX,\DD)$ is at most $k+5$-dimensional and it attains maximal possible dimension if and only if $(\XX,\DD)$ is equivalent to the regular pair corresponding to the trivial equation: $x^{(k+1)}=0$.  The following structural equations are satisfied:
\begin{eqnarray}
&&[\BG,\BX]=0,\label{s1}\\
&&[\BF^0,\BX]=-\BX,\label{s2}\\
&&[\BF^1,\BX]=-2\BF^0-k\BG,\label{s3}\\
&&[\BG,\BV^i]=\BV^i,\label{s4}\\
&&[\BF^0,\BV^i]=-i\BV^i,\label{s5}\\
&&[\BF^1,\BV^j]=i(i-1-k)\BV^{i-1}\mod \BX,\BG,\BF.\label{s6}
\end{eqnarray}
\end{theorem}
\begin{proof}
Let us first introduce local coordinates on $E(\XX,\DD)$. Let $X$ be a projective vector field and let $V$ be a corresponding normal section of $\DD$. Let $j^1X(x)$ denotes a 1-jet of $X$ at $x\in M$ and let $\widetilde{V}(x)= V(x)\mod \XX$ be an element of the quotient space $\DD(x)/\XX(x)$. Then arbitrary $p\in E(\XX,\DD)(x)$ can be uniquely written in the form $p=R_{(G,F)}(\widetilde{V}(x), j^1X(x))$ for certain $G\in Gl(1)$ and
$$
F=\left(\begin{array}{cc}
F_0&F_1\\
0&1\\
\end{array}\right)\in Aff(1),
$$
where $R_{(G,F)}$ stands for the right action of the group on the bundle. In this way local coordinates $(G,F)$ on fibres of $E(\XX,\DD)$ are introduced. Using the coordinates we can treat $X$ and $V$ as vector fields on $E(\XX,\DD)$, not only on $M$. Moreover, it is straightforward to see that in the coordinates
$$
\BG=G\partial_G,\qquad \BF^0=F_0\partial_{F_0},\qquad \BF^1=F_0\partial_{F_1}.
$$
Additionally we have
\begin{lemma}\label{l1}
$$
\BX=\frac{1}{F_0}X -2\frac{F_1}{F_0}\BF^0 -\frac{(F_1)^2}{(F_0)^2}\BF^1 -k\frac{F_1}{F_0}\BG.
$$
\end{lemma}
\begin{proof}
Let us assume that $X=\partial_t$ is a projective vector field and let $f$ and $g$ be functions from Proposition \ref{p1}. Then $x_0=f$, $x_1=\dot f$ and $y=g$ can be taken as coordinates on fibres of $E(\XX,\DD)$. Then, equations \eqref{eq2} and \eqref{eq3} imply that a partial Ehresmann connection $\widetilde{\XX}$ on $E(\XX,\DD)$ is given by the formula
$$
\widetilde{\XX}=\spn\left\{\partial_t+x_1\partial_{x_0} +\frac{x_1^2}{2x_0}\partial_{x_1} -\frac{1}{2}\frac{x_1}{x_0}y\partial_y\right\},
$$
since $\dot g=-\frac{k}{2}\frac{\dot f}{f}g$ and $\ddot f=\frac{\dot f^2}{2f}$. We are going to change coordinate systems from $(x_0,x_1,y)$ to $(F_0,F_1,G)$. First of all $G=y$. Moreover, if $t'=\frac{F_0t}{F_1t+1}$ is a parameter corresponding to the projective vector field $fX$ then on the submanifold $\{t=0\}$
$$
f(0)=\frac{1}{F_0},\qquad \dot f(0)=\frac{2F_1}{F_0}.
$$
Thus $(F_0,F_1)$ and $(x_0,x_1)$ are related by the diffeomorphism $\Phi\colon(F_0,F_1)\mapsto\left(\frac{1}{F_0},\frac{2F_1}{F_0}\right)$. Computing $D\Phi^{-1}$ and applying it to $\widetilde{\XX}$ we get
$$
\widetilde{\XX}=\spn\left\{X-2F_1\BF^0 -\frac{(F_1)^2}{F_0}\BF^1 -kF_1\BG\right\}.
$$
Taking into account the condition $\pi_*(\BX(p))=\pi_X(p)$ we get the desired formula for $\BX$ and Lemma \ref{l1} is proved.
\end{proof}

The rest of the proof of Theorem \ref{t2} is divided into three parts. At the beginning we will construct canonical frame, then we will show that structural equations are satisfied and finally we will consider the most symmetric case.
\vskip 1ex
{\bf Construction of the canonical frame.}
We will work in local coordinates as before. We can write
$$
\BV^0=GV+\alpha\BX+\beta\BG+\gamma_0\BF^0+\gamma_1\BF^1.
$$
Our aim is to prove that functions $\alpha,\beta,\gamma_0,\gamma_1$ are uniquely defined by \eqref{eq4}.

We use Lemma \ref{l1} and compute that
\begin{eqnarray*}
\BV^1&=&[\BX,\BV^0_j]\\
&=&\frac{G}{F_0}(\ad_XV-kF_1V)+(\BX(\alpha)+\gamma_0)\BX
+\left(\BX(\beta)+k\gamma_1\right)\BG \mod \BF.
\end{eqnarray*}
Then, by induction, we get
$$
\BV^i=\frac{G}{(F_0)^i}\left(\ad^i_XV+c^i_{i-1}F_1\ad^{i-1}_XV +\cdots+c^i_1F_1^{i-1}\ad_XV+c^i_0F_1^iV\right)\mod \BX,\BG,\BF,
$$
for $i=2,\ldots,k$ where $c^i_j$ are certain rational numbers which exact values are not important for us. We compute that
\begin{eqnarray*}
[\BV^0,\BV^1]&=&\frac{G^2}{F_0}[V,\ad_XV]-(\BX(\beta)+2k\gamma_1)\BV^0\\
&-&(\BX(\alpha)-\beta+2\gamma_0)\BV^1+\alpha\BV^2 \mod\BX,\BG,\BF,
\end{eqnarray*}
and
\begin{eqnarray*}
[\BV^0,\BV^3]&=&\frac{G^2}{(F_0)^3}\left( \sum_{j=0}^3c^3_jF_1^{3-j}[V,\ad^j_XV]\right)\\
&+&(\beta-3\gamma_0)\BV^3+\alpha\BV^4\mod\BV^0,\BV^1,\BV^2,\BX,\BG,\BF,
\end{eqnarray*}
where in above we assumed that $k\geq 4$. If $k=3$ then the term $\alpha\BV^4$ is replaced by $\tilde c\alpha\BV^3$ for a certain function $\tilde c$ on $E(\XX,\DD)$, but it does not change the reasoning below.

Since vector fields $\ad^i_XV$, $i=0,\ldots,k$, together with $X$ span the whole tangent bundle $TM$, we can express the Lie bracket $[V,\ad^j_XV]$ in terms of $\ad^i_XV$ and $X$. However, we can also write $[V,\ad^j_XV]=\sum_ic_i^j\BV^i_j\mod \BX,\BG,\BF$, for some functions $c_i^j$ on $E(\XX,\DD)$. Then we are able to rewrite condition \eqref{eq4} in terms of unknown functions  $\alpha$, $\beta$, $\gamma_0$ and $\gamma_1$. We get the following system of equations
\begin{eqnarray}
\BX(\beta)+2k\gamma_1&=&C_1,\label{u1}\\
\BX(\alpha)+2\gamma_0-\beta&=&C_2,\label{u2}\\
\alpha&=&C_3,\label{u3}\\
\beta-3\gamma_0&=&C_4\label{u4}
\end{eqnarray}
where $C_1$, $C_2$, $C_3$ and $C_4$ are certain functions on $E(\XX,\DD)$ (in the case $k=3$ equation \eqref{u4} takes the form $\beta-3\gamma_0+\tilde c\alpha=C_4$). Now, the system \eqref{u1}-\eqref{u4} can be easily solved for $\alpha$, $\beta$, $\gamma_0$ and $\gamma_1$ and the solution is unique. This proves the first part of Theorem \ref{t2}.
\vskip 1ex
{\bf Structural equations.}
At the beginning we use Lemma \ref{l1} and directly compute that \eqref{s1}-\eqref{s3} are satisfied. To prove \eqref{s4}-\eqref{s6} it is sufficient to show that they are satisfied for $i=0$. Then, we consider $[\BX,[\BG,\BV^i]]$, $[\BX,[\BF^0,\BV^i]]$ and $[\BX,[\BF^1,\BV^i]]$ and apply Jacobi identity. By a simple induction, using \eqref{s1}-\eqref{s3} we get \eqref{s4}-\eqref{s6} in full generality. For example, taking $[\BX,[\BF^1,\BV^i]]$, on the one hand we get
$$
[\BX,[\BF^1,\BV^i]]=[\BF^1,\BV^{i+1}]+[2\BF^0+k\BG,\BV^i] =[\BF^1,\BV^{i+1}]+(k-2i)\BV^i\mod\BX,\BG,\BF
$$
and on the other hand we get
$$
[\BX,[\BF^1,\BV^i]]=i(i-1-k)[\BX,\BV^{i-1}]=i(i-1-k)\BV^i\mod\BX,\BG,\BF.
$$
Combining the two expressions above we get
$$
[\BF^1,\BV^{i+1}]=(i(i-1-k)-k+2i)\BV^i=(i+1)(i-k)\BV^i\mod\BX,\BG,\BF
$$
as required.

A validation of relation \eqref{s6} for $i=0$ is immediate taking into account the formula for $\BV^0$.  Moreover, it is also clear, that \eqref{s4} and \eqref{s5} also hold modulo $\BX$, $\BG$ and $\BF$. Thus we can assume that $[\BG,\BV^0]=\BV^0+\Phi$ and $[\BF^0,\BV^0]=\Psi$ for some $\Phi,\Psi\in\spn\{\BX,\BG,\BF\}$. We shall show that $\Phi=\Psi=0$. For this, let us notice first that condition \eqref{eq4} is equivalent to
\begin{eqnarray}
&&[\BV^0,\BV^1]=0 \mod\BV^3,\ldots,\BV^k,\BX,\BG,\BF,\label{eq5}\\
&&[\BV^0,\BV^3]=0 \mod\BV^0,\BV^1,\BV^2,\BV^4,\ldots,\BV^k,\BX,\BG,\BF\label{eq6}
\end{eqnarray}
and we already know that Lie bracket with $\BG$ and $\BF^0$ preserve the right hand side of \eqref{eq5} and \eqref{eq6}. Therefore it follows that
\begin{eqnarray}
&&[\BG,[\BV^0,\BV^1]]=0\mod\BV^3,\ldots,\BV^k,\BX,\BG,\BF,\label{eq7}\\
&&[\BF^0,[\BV^0,\BV^1]]=0 \mod\BV^3,\ldots,\BV^k,\BX,\BG,\BF\label{eq7b}
\end{eqnarray}
and
\begin{eqnarray}
&&[\BG,[\BV^0,\BV^3]]=0 \mod\BV^0,\BV^1,\BV^2,\BV^4,\ldots,\BV^k,\BX,\BG,\BF\label{eq8}\\
&&[\BF^0,[\BV^0,\BV^3]]=0 \mod\BV^0,\BV^1,\BV^2,\BV^4,\ldots,\BV^k,\BX,\BG,\BF.\label{eq8b}
\end{eqnarray}
Now, if we expand the left hand sides of \eqref{eq7}-\eqref{eq8b} and use \eqref{eq5}-\eqref{eq6} we get that both $\widetilde{\BV}^i=\ad^i_\BX(\BV^0+\Phi)$ and $\widehat{\BV}^i=\ad^i_\BX(\BV^0+\Psi)$ satisfy condition \eqref{eq4}.

For instance, Jacobi identity applied twice to \eqref{eq7b} implies
$$
[\Psi,\BV^1]+[\BV^0,\BV^1+[\BX,\Psi]]=0\mod\BV^3,\ldots,\BV^k,\BX,\BG,\BF
$$
and since $[\Psi,[\BX,\Psi]]=0\mod\BX,\BG,\BF$ we get
$$
[\widehat{\BV}^0,\widehat{\BV}^1]=0\mod\BV^3,\ldots,\BV^k,\BX,\BG,\BF.
$$
Similarly, applying Jacobi identity to \eqref{eq8b}, taking into account the relation $[\BF^0,\BV^3]=3\BV^3\mod\BX,\BG,\BF$ and using \eqref{eq6} we get
$$
[\Psi,\BV^3]=0\mod\BV^0,\BV^1,\BV^2,\BV^4,\ldots,\BV^k,\BX,\BG,\BF.
$$
Then it follows that
$$
[\widehat{\BV}^0,\widehat{\BV}^3]=0 \mod\BV^0,\BV^1,\BV^2,\BV^4,\ldots,\BV^k,\BX,\BG,\BF
$$
and thus $\widehat{\BV}^i$ satisfy \eqref{eq4} (the reasoning for $\widetilde{\BV}^i$ is analogous).

Now, the two frames: $(\BG,\BF^0,\BF^1,\BX,\widetilde{\BV}^0,\ldots,\widetilde{\BV}^k)$ and $(\BG,\BF^0,\BF^1,\BX,\widehat{\BV}^0,\ldots,\widehat{\BV}^k)$ are adapted. But we have already proved that there is a unique adapted frame satisfying \eqref{eq4}. Therefore $\Phi=\Psi=0$.
\begin{remark}
The Lie bracket with $\BF^1$ does not preserve the right hand side of \eqref{eq5} and \eqref{eq6}. Therefore equation \eqref{s6} is satisfied modulo $\BX$, $\BG$ and $\BF$ only.
\end{remark}

{\bf Uniqueness of the model with maximal symmetry group.}
If the dimension of the symmetry group is maximal possible then structural functions of the canonical frame are constant. It follows from the structural equations that we only have to consider Lie brackets involving $\BV^i$ in order to determine when all structural functions are constant. Note that vector fields $\BV^i$ are linear in $G$. On the other hand $\BX$, $\BG$, $\BF$ are homogeneous of order 0 in $G$. It follows that structural functions of the Lie brackets $[\BV^i,\BV^j]$ are either homogeneous of order one or two. Thus, in the most symmetric case all of them vanish. Similarly the structural functions next to $\BX$, $\BG$ and $\BF$ of Lie brackets $[\BF^0,\BV^i]$ and $[\BF^1,\BV^i]$ vanish and the remaining, possibly non-trivial, structural functions are those next to $\BV^j$ for the Lie bracket $[\BX,\BV^k]$. We have
$$
[\BX,\BV^k]= w_0\BV^0+\cdots+w_{k}\BV^{k},
$$
for some $w_i$. We claim that all $w_i$ are homogeneous of order $k+1-i$ in $F_0$. Indeed in local coordinates on $E(\XX,\DD)$ the vector fields $\BV^i$ are homogeneous of order $-i$ in $F_0$ (at least modulo $\BX,\BG,\BF$). Additionally we see that $[\BX,\BV^k]$ is of order $-(k+1)$. Thus $w_i$ is homogeneous of order $k+1-i$. As a conclusion we get that all $w_i$ vanish provided that they are constant. In this way we have proved uniqueness of the most symmetric model of regular pairs $(\XX,\DD)$. On the other hand it is well known that the trivial system $x^{(k+1)}=0$ has group of contact symmetries of dimension equal to $\dim E(\XX,\DD)$. In this way the proof of Theorem \ref{t2} is completed.
\end{proof}

As a direct consequence of Theorem \ref{t2} and Theorem \ref{t1} we get the solution to the problem of equivalence of $Gl(2)$-structures. In particular we get that a $Gl(2)$ on $k+1$-dimensional manifold has $k+5$-dimensional algebra of infinitesimal symmetries if and only if it is flat.

\section{Appendix. Remarks on ODEs}\label{sec5}
In Appendix we will provide some more remarks on our canonical frame in the case of regular pairs of equation type. In \cite{K1} we have proved that a pair $(\XX,\DD)$ is of equation type (in the sense of Section \ref{sec2}) if $\rk\,\DD^{(i)}=i+1$, $i=0,\ldots,k$, where $\DD^{(1)}=\DD$ and
$$
\DD^{(i+1)}=[\DD^{(i)},\DD^{(i)}]
$$
i.e.~$\DD$ is Goursat distribution. Since $\ad_\XX^{i-1}\DD\subset \DD^{(i)}$ it follows from (G1)-(G2) that $\ad_\XX^{i-1}\DD=\DD^{(i)}$.

Now, for a general pair $(\XX,\DD)$, let us denote $\widetilde{\DD}=\pi^{-1}_*(\DD)$, where as before $\pi\colon E(\XX,\DD)\to M$ is the projection. Then, since $\widetilde{\DD}=\{\BV^0,\BX,\BG,\BF\}$, it follows from \eqref{s1}-\eqref{s6} that
$$
\widetilde{\DD}^{(i)}=\pi^{-1}_*(\DD^{(i)})
$$
and we get the following characterisation of regular pairs of equation type in terms of torsion invariants $T^{pq}_r$.
\begin{proposition}\label{p2}
If $(\XX,\DD)$ is of equation type then
$$
\widetilde{\DD}^{(i)}=\ad^{i-1}_\BX\widetilde{\DD} =\spn\{\BV^0,\ldots,\BV^{i-1},\BX,\BG,\BF\}.
$$
In particular $T^{pq}_r=0$ for $r>\max\{p,q\}+1$. Conversely, if $T^{pq}_r=0$ for $r>\max\{p,q\}+1$ then $\widetilde{\DD}^{(i)}=\ad^{i-1}_\BX\widetilde{\DD}$ and $(\XX,\DD)$ is of equation type.
\end{proposition}

For regular pairs of equation type we can strengthen Theorem \ref{t2}.
\begin{proposition}\label{p3}
If $(\XX,\DD)$ is of equation type then
\begin{equation}\label{s6b}
[\BF^1,\BV^j]=i(i-1-k)\BV^{i-1}.
\end{equation}
\end{proposition}
\begin{proof}
If $(\XX,\DD)$ is of equation type then, since $[\BV^0,\BV^1]\in\widetilde{D}^{(3)}$ we get that condition \eqref{eq5} is equivalent to
\begin{equation}\label{eq9}
[\BV^0,\BV^1]=0\mod\BX,\BG,\BF.
\end{equation}
Applying twice the Lie bracket with $\BX$ to both sides of the above equation we get
$$
[\BV^0,\BV^2]=0\mod\BX,\BG,\BF
$$
and 
$$
[\BV^0,\BV^3]+[\BV^1,\BV^2]=0\mod\BX,\BG,\BF.
$$
The last identity implies that $[\BV^0,\BV^3]=-[\BV^1,\BV^2]\mod\BX,\BG,\BF$. But $[\BV^1,\BV^2]\in\widetilde{D}^{(4)}$ and thus $[\BV^0,\BV^3]\in\widetilde{D}^{(4)}$. Therefore, condition \eqref{eq6} takes the form
\begin{eqnarray}\label{eq10}
[\BV^0,\BV^3]=0\mod\BV^0,\BV^1,\BV^2,\BX,\BG,\BF.
\end{eqnarray}
Now, note that the right hand sides of both \eqref{eq9} and \eqref{eq10} are invariant with respect to taking Lie bracket with $\BF^1$. Therefore we can apply the same reasoning as in the proof of \eqref{s4} and \eqref{s5} in Section \ref{sec4} and in this way we get \eqref{s6b}.
\end{proof}

Let us denote
$$
\BH=2\BF^0+k\BG,\qquad \BY=\BF^1,
$$
and
$$
\BW^i=\frac{1}{i!}\BV^i
$$
Then we have easily compute
\begin{equation}
\begin{split}
&[\BX,\BY]=\BH,\qquad [\BH,\BX]=-2\BX,\qquad [\BH,\BY]=2\BY,\\
&[\BG,\BX]=0,\qquad [\BG,\BY]=0,\qquad [\BG,\BH]=0,\\
\end{split}\nonumber
\end{equation}
and, if $(\XX,\DD)$ is of equation type, then
\begin{equation}
\begin{split}
&[\BX,\BW^i]=(i+1)\BW^{i+1},\qquad [\BY,\BW^i]=-(k-i+1)\BW^{i-1},\\
&[\BH,\BW^i]=-2i\BW^i,\qquad [\BG,\BW^i]=\BW^i.
\end{split}\nonumber
\end{equation}
It follows from above that the coframe on $E(\XX,\DD)$ dual to $(\BG,\BH,\BX,\BY,\BW^0,\ldots,\BW^k)$ is a Cartan connection of type $(T(2),G)$ with $G$ being a semidirect product of $Gl(2)$ and $\R^{k+1}$ where $Gl(2)$ acts irreducibly on $\R^{k+1}$.

\end{document}